\documentclass[reqno]{amsart}     
\usepackage{enumerate,amsmath,amssymb}     
\usepackage{pstricks,pst-node,pst-coil,pst-plot}   
\usepackage{graphicx}

\theoremstyle{plain}      
 
\newtheorem{thm}{Theorem}[section]     
\newtheorem{theorem}[thm]{Theorem}     
\newtheorem{cor}[thm]{Corollary}

\newtheorem{lemma}[thm]{Lemma}     
\newtheorem{prop}[thm]{Proposition}

\theoremstyle{remark}

\theoremstyle{definition}      
     
\newtheorem{definition}[thm]{Definition}     

\def\al{{\alpha}}         
         
\def\de{{\delta}}

\def\Om{{\Omega}}         
\def\la{{\lambda}}

\def\Si{{\Sigma}}         
         
\def\ep{{\epsilon}}

\let\theta\vartheta
\def\Th{{\Theta}}         
         
\def\phi{{\varphi}}

\DeclareMathAlphabet{\doba}{U}{msb}{m}{n}

\gdef\mN{\doba{N}}

\gdef\mR{\doba{R}}

\def\Scal{{\mathop{\rm Scal}}}     

\def\Ric{{\mathop{\rm Ric}}}

\def\eref#1{{\rm (\ref{#1})}}   

\let\<\langle 
\let\>\rangle

\newcommand{\definedas}{\mathrel{\raise.095ex\hbox{\rm :}\mkern-5.2mu=}}

\def\tgp{\Theta^{g}_p}
\def\sch{{\rm S}}

\begin{document}     


\title{An isoperimetric constant associated to horizons 
in $S^3$ blown-up at two points.}  


\author{Mattias Dahl} 
\address{Institutionen f\"or Matematik \\
Kungliga Tekniska H\"ogskolan \\
100 44 Stockholm \\
Sweden}
\email{dahl@math.kth.se}

\author{Emmanuel Humbert} 
\address{Institut \'Elie Cartan, BP 239 \\ 
Universit\'e de Nancy 1 \\
54506 Vandoeuvre-l\`es-Nancy Cedex \\ 
France}
\email{ehumbert@iecn.u-nancy.fr}

\begin{abstract}
Let $g$ be a metric on $S^3$ with positive Yamabe constant. When blowing
up $g$ at two points, a scalar flat manifold with two asymptotically
flat ends is produced and this manifold will have compact minimal surfaces. We
introduce the $\Th$-invariant for $g$ which is an isoperimetric
constant for the cylindrical domain inside the outermost minimal
surface of the blown-up metric. Further we find relations between $\Th$
and the Yamabe constant and the existence of horizons in the blown-up
metric on $\mR^3$.
\end{abstract}     


\date{\today}


\maketitle     

\tableofcontents

\section{Introduction}

Let $(N,h)$ be a $3$-manifold with an asymptotically flat end. An
outermost minimal surface is a compact minimal surface which encloses
all other compact minimal surfaces. As long as $N$ is not
diffeomorphic to $\mR^3$, a result due to Meeks, Simon, and Yau
\cite{meeks_simon_yau_82} guarantees the existence of a compact
minimal surface. Using the asymptotic flatness one then finds an
outermost minimal surface.

Let $(M,g)$ be a compact Riemannian $3$-manifold with positive
Yamabe constant and fix $p \in M$. We denote by $G_p$ the Green's
function at $p$ for the Yamabe operator. The manifold 
$(M \setminus \{ p \}, G_p^4 g)$ is asymptotically flat and scalar
flat. If $M$ is not diffeomorphic to $S^3$ then 
$M^3 \setminus \{ p \}$ is not diffeomorphic to $\mR^3$ and hence the
result mentioned above gives the existence of an outermost minimal
surface in $(M \setminus \{ p \}, G_p^4 g)$. If $M = S^3$ the
existence of an outermost minimal surface in 
$(S^3 \setminus \{p \}, G_p^4 g)$ depends on $g$. For instance, if $g$
is the standard round metric of $S^3$ then the corresponding
asymptotically flat metric is $\mR^3$ equipped with its standard
Euclidean metric and hence does not possess any compact minimal
surface. On the other hand, if $g$ is close enough to a scalar flat
metric, then the corresponding asymptotically flat metric will have an
(outermost) minimal surface,
see \cite{beig_omurchadha_91}, \cite{yan_05}, and
Section \ref{metrics_large}. To characterize the metrics $g$ on $S^3$
for which $(M \setminus \{ p \}, G_p^4 g)$ have a minimal surface is
an open problem.

One the contrary, if $g$ est a metric on $S^3$ blown-up at two points, then it always contains an horizon. In other words, if $g$ is a
metric on $S^3$, and if $p,q \in S^3$ are distinct points of $S^3$
then $(S^3 \setminus \{p,q\}, (G_p+G_q)^4 g)$ is asymptotically flat
and scalar flat but possesses an outermost minimal surface since $S^3
\setminus \{p,q \}$ is not diffeomorphic to $\mR^3$. The existence of
this outermost minimal surface allows us to apply powerful tools such
as the weak inverse mean curvature flow developed by Huisken and
Ilmanen \cite{huisken_ilmanen_01}.

In this paper, we define the invariant $\Theta$
by  
$$
\Theta_p^g(q) \definedas \frac{|\Om|}{|\Si|^{3/2}},
$$ 
where $\Si$ is the only outermost minimal surface in 
$(S^3 \setminus \{p,q\}, (G_p+G_q)^4 g)$ bounding a cylindrical domain 
$\Om$ diffeomorphic to $S^2 \times (a,b)$ for some $a,b \in \mR$, 
$a \leq b$. Here the volume of $\Om$ and the area of $\Si$ are
computed in the metric $(G_p+G_q)^4 g$. We show that
$\tgp$ has several interesting properties, in particular it is related
to the Yamabe constant of $g$. 

Beside these interesting properties, the motivation for studying such an isoperimetric quotient comes from the following observation : the metric $(G_p+G_q)^4 g$ tends to
$16 G_p^4 g$ in all $C^k$ on all compact sets $K \subset S^3 \setminus
\{ p\}$ when $q$ tends to $p$. It then seems natural to study such
metrics blown up in two points to get information on the metrics blown
up in one point. We expect to get results of this kind by studying the behavior of $\tgp(q)$ as $q$ tends to $p$.

\section{Preliminaries}

In this section we recall some well-known facts about asymptotically
flat $3$-manifolds, the inverse mean curvature flow, and the Yamabe
operator. We begin by establishing some notational conventions.

The standard euclidean metric on $\mR^3$ is denoted by $\xi$ and the 
round metric on $S^3$ of constant sectional curvature $1$ is denoted
by $\sigma$. For a Riemannian manifold $(M,g)$ with a point $p \in M$
we denote by $B^g_p(\delta)$ the open ball of all points of distance
less than $\delta$ to $p$. The gradient of a function $u$ is denoted
by $\nabla^g u$ or $\nabla u$, since it usually only appears in norm 
$|\nabla u|_g$ there is no risk of confusion when omitting the
Riemannian metric from the notation. For an open subset $\Omega$ in
the Riemannian $3$-manifold $(M,g)$ we denote the volume by
$|\Omega|_g$ and for a surface $\Sigma$ in $M$ we denote the area by
$|\Sigma|_g$.

\subsection{Asymptotically flat $3$-manifolds}

\begin{definition}
Let $(M,g)$ be a Riemannian $3$-manifold. 
\begin{itemize} 
\item 
An {\it asymptotically flat end} of $(M,g)$ is an open set $E$ of $M$
diffeomorphic to the complement of a compact set in $\mR^3$. In the
coordinates given by this diffeomorphism the metric $g$ is required to
satisfy  
$$
|g_{ij} - \xi_{ij}| \leq \frac{C}{|x|}, 
\quad 
|\partial_k g_{ij} | \leq \frac{C}{|x|^2}, 
\quad  
\Ric^g \geq \frac{-C}{|x|^2} g,
$$ 
for large $|x|$. 
\item 
The Riemannian manifold $(M,g)$ is said to be {\it asymptotically 
flat} if $(M,g)$ with a compact set removed is a union of
asymptotically flat ends. 
\end{itemize} 
\end{definition}

The simplest example of an asymptotically flat manifold is
$(\mR^3, \xi)$ which has one end. Another example which plays a central role
in many problems is the {\it spatial Schwarzschild manifold}
defined by
\begin{equation} \label{schwarzschild_def}
(\sch,g_{\sch})
\definedas
\left(
\mR^3 \setminus \{0 \}, 
\left(1 + \frac{m}{2 |x|} \right)^4 \xi 
\right).
\end{equation}
This is an asymptotically flat manifold with two ends. Note that it
possesses an involutive isometry fixing the sphere of radius $m/2$
(with respect to $\xi$) centered at the origin. In the
Schwarzschild metric this sphere has area $16 \pi m^2$.  

Let $E$ be an asymptotically flat end of $(M,g)$. Then, Arnowitt,
Deser, and Misner \cite{arnowitt_deser_misner_61} introduced the ADM
mass given by
$$
m^g(E) \definedas 
\lim_{r \to \infty} 
\frac{1}{16 \pi} \int_{S_r} (\partial_j g_{ii} - \partial_i g_{ij})
\nu^j \, da^{\xi} ,
$$ 
where $S_r$ is the sphere centered at the origin and of radius $r$ in
$\mR^3$ and where $da^{\xi}$ is the area element induced by $\xi$
on $S_r$. This quantity does not depend on the coordinates and is
finite when 
\begin{equation} \label{scal_bounded}
\int_E |\Scal^g| \, dv^g < \infty.
\end{equation}
See for instance Bartnik \cite{bartnik_86} for further discussion. A
fundamental result concerning the mass is the Positive Mass Theorem. 
\begin{theorem}
Let $(M,g)$ be an asymptotically flat $3$-manifold whose scalar
curvature is non-negative and satisfies \eref{scal_bounded}. Then
$$
m^g(E) \geq 0
$$
for each end $E$ with equality if and only if $(M,g)$ is isometric to
$(\mR^3, \xi)$. 
\end{theorem} 

This theorem was first proved by Schoen and Yau \cite{schoen_yau_79},
notable among the other proofs available is the one of 
Witten \cite{witten_81} which uses spin geometry.



A compact minimal surface in an asymptotically manifold is called an
{\it horizon}. A minimal surface is called {\it outermost}
\cite{bray_01} if it is not contained entirely inside another
minimal surface. 



\subsection{Inverse mean curvature flow}

In their proof of the Penrose inequality \cite{huisken_ilmanen_01},
Huisken and Ilmanen introduced the "weak inverse mean curvature
flow". The standard inverse mean curvature flow may develop 
singularities and is therefore difficult to use. On the contrary,
the weak inverse mean curvature flow gives a flow for "almost all
$t$" and provides a powerful technique in many situations. As an
example, Bray and Neves \cite{bray_neves_04} used this tool to show
that the Yamabe constant of $\mR P^3$ is attained by the constant
curvature metric. In Section \ref{upper_bound} we will use the method
of Bray and Neves to prove Theorem \ref{theta_leq}. We recall some
basic facts about the weak inverse mean curvature flow. First, if
$\Si$ is a $C^1$ surface of a Riemannian $3$-manifold $(N,h)$, we say
that $H \in L^1_{loc}(\Si)$ is the {\it weak mean curvature} of $\Si$
if
$$
\int_\Si {\mathop{\rm div}}^h (X) \, da^h 
= 
\int_\Si H h(X,\nu) \, da^h
$$ 
for all compactly supported vector fields $X$, where $\nu$
is the outer normal vector field on $\Si$. This definition coincides
with the usual one as soon as $\Si$ is smooth.  

\begin{definition}
Let $\Si$ be a compact $C^1$ hypersurface $\Si$ with weak mean
curvature $H$ in $L^2(\Si)$. The {\it Hawking mass} of $\Si$ is
defined by
$$
m_H(\Si) 
\definedas
\sqrt{\frac{|\Si|_h}{(16 \pi)^3}} 
\left(16 \pi - \int_\Si H^2 \, da^h \right).
$$
Here $|\Si|_h$ is the area of $\Si$ computed using the metric $h$. 
\end{definition}

We collect the main properties of the inverse mean curvature flow as
in \cite[Theorem 5.2]{bray_neves_04}.

\begin{theorem}\label{mcf}
\cite{huisken_ilmanen_01}
Let $(N,h)$ be an asymptotically flat $3$-manifold with 
non-negative
scalar curvature. We assume that $N$ is diffeomorphic to 
$\mR^3 \setminus B$ where $B$ is the unit ball in $\mR^3$ and that 
$\partial N=S^2$ is an outermost horizon. Then, there exists a
precompact locally Lipschitz function $\Phi$ satisfying
\begin{itemize}
\item 
for all $t \geq 0$, $\Si_t \definedas \partial \{ \Phi < t \}$ defines
an increasing family of $C^{1,\al}$ surfaces such that $\Si_0=\Si$; 
\item 
for almost all $t \geq 0$, the weak mean curvature of $\Si_t$ is 
$|\nabla \Phi|_h$;
\item 
for almost all $t \geq 0$,  $|\nabla \Phi|_h \neq 0$ on $\Si_t$ 
for almost all $x \in \Si_t$ (with respect to the surface measure) and  
\begin{equation*}
|\Si_t|_h = |\Si_0|_h e^t
\end{equation*}
for all $f \geq 0$;
\item 
The Hawking mass $m_H(\Si_t)$ is a non-decreasing function of 
$t \geq 0$ provided the Euler characteristic $\chi(\Si_t) \leq 2$ for
all $t \geq 0$.   
\end{itemize} 
\end{theorem}

\subsection{The Yamabe operator and the Green's function} 
\label{green}

Let $g$ be a Riemannian metric on the $3$-sphere $S^3$. We set
$$
L^g \definedas 8 \Delta^g + \Scal^g.
$$
This self-adjoint elliptic operator is called the 
{\it Yamabe operator} and is conformally invariant in the following
sense. If $h = u^4 g$ where $u$ is a smooth positive function is a
metric conformal to $g$ then the Yamabe operators of $g$ and $h$ are
related by 
$$
L^h f = u^{-5} L^{g} (uf)
$$
and the scalar curvature of $h$ is given by 
\begin{equation} \label{rel_scal} 
\Scal^h = u^{-5} L^g u. 
\end{equation} 
The Yamabe constant of the metric $g$ is defined by 
$$
\mu(g) 
\definedas
\inf_{u\in C^\infty(S^3); u \neq 0} 
\frac{\int_{S^3} u L^g u \, dv^g}
{\left( \int_{S^3} u^6 \, dv^g \right)^{1/3} }
=
\inf_{u\in C^\infty(S^3); u \neq 0} 
\frac{ \int_{S^3} (8 |\nabla u|_g^2 + \Scal^g u^2) \, dv^g}
{\left( \int_{S^3} u^6 \, dv^g \right)^{1/3} } .
$$
The number $\mu(g)$ is conformally invariant and it is known that
$\mu(g) > 0$ (resp. $\mu(g) = 0$, resp. $\mu(g) < 0$) if and only if
there exists a metric in the conformal class of $g$ with positive
(resp. identically zero, resp. negative) scalar curvature.
 
Assume from now on that the metric $g$ has a positive Yamabe
constant. Then $L^g$ is invertible and if $p \in S^3$ is a
fixed point, this allows to construct the unique 
{\it Green's function} $G_p$ for $L^g$ at $p$, see 
\cite[Lemma 6.1]{lee_parker_87}. We recall that $G_p$ is smooth on
$S^3 \setminus \{ p \}$, satisfies 
\begin{equation} \label{eqG} 
L^g G_p = \delta_p
\end{equation}
in the sense of distributions, and has the expansion
\begin{equation} \label{expan}
G_p = \frac{1}{d^g(p,\cdot)} + \alpha_p 
\end{equation}
at $p$, where $\alpha_p$ is a smooth function defined in a
neighborhood of $p$. Set $S_p \definedas S^3 \setminus \{ p \}$ and 
$g_p \definedas G_p^4 g$. Then the Riemannian manifold $(S_p, g_p)$ is
asymptotically flat with one end $E_p = S_p$, by \eref{rel_scal} and
\eref{eqG} it is scalar flat, and one can deduce from \eref{expan}
that its mass is given by \cite[Lemma 9.7]{lee_parker_87}
\begin{equation*} 
m^{g_p}(E_p) = \alpha_p(p).
\end{equation*} 
As an example, if $g$ is the round metric on $S^3$, one easily
checks that $(S_p,g_p)$ is isometric to $\mR^3$ equipped with its
standard Euclidean metric.

A question which we will interest us here is the existence of
horizons for the metric $g_p$, that is of compact minimal surfaces in
the Riemannian manifold $(S_p,g_p)$. First note that the same
construction on any $3$-manifold not diffeomorphic to $S^3$ always
gives rise to a horizon. Note also that, as an application of
the techniques from \cite{bray_neves_04}, Miao finds that a
necessary condition for the existence of a horizon is that
$\mu(g) \leq \mu(\sigma)/2^{2/3} $ where $\sigma$ is the round metric
on $S^3$, see \cite{miao07}.

Now, let us define $g_{p,q} \definedas (G_p + G_q)^4 g$ and 
$S_{p,q} \definedas S^3 \setminus \{p,q\}$ for $p,q \in S^3$. Then
$(S_{p,q}, g_{p,q})$ is also asymptotically flat and scalar flat but
has two ends $E_p$ and $E_q$, whereas $(S_p,g_p)$ has only one end. As
an example, if $g = \sigma$ is the round metric on $S^3$ and if
$q=-p$, then $(S_{p,q},g_{p,q})$ is isometric to the Schwarzschild
manifold $(\mR^3 \setminus \{0 \}, (1 + |x|^{-1} )^4 \xi)$. Since 
$$
G_p+G_q 
= 
\frac{1}{d^g(p,\cdot)} + \alpha_p + G_q
$$
near $p$ and since 
$$ 
G_p+G_q 
= 
\frac{1}{d^g(q,\cdot)} + \alpha_q + G_p
$$
near $q$, one checks that the masses $m^{g_{p,q}} (E_p)$ and 
$m^{g_{p,q}} (E_q)$ of the manifold $(S_{p,q}, g_{p,q})$ at the ends 
$E_p$ and $E_q$ are given by \cite[Lemma 9.7]{lee_parker_87}
\begin{equation} \label{mass_pq}
m^{g_{p,q}} (E_p) = \alpha_p(p) + G_q(p) 
\quad \text{and} \quad 
m^{g_{p,q}} (E_q) = \alpha_q(q) + G_p(q). 
\end{equation}

Another important difference compared to the case of one blow-up point
is that $(S_{p,q},g_{p,q})$ has always an horizon. We deal with the
outermost horizon. More precisely, there exists a compact minimal
surface $\tilde{\Si}_{p,q}$ (not necessarily connected) in $(S_{p,q},g_{p,q})$
bounding a bounded domain $\tilde{\Om}_{p,q}$ (maybe empty) such that any
other compact horizon lies inside $\tilde{\Om}_{p,q}$. One sees 
that the connected components of $\tilde{\Si}_{p,q}$ are diffeomorphic to
$S^2$ (see Lemma 4.1 in \cite{huisken_ilmanen_01}). At least one and
at most two of them divide $S_{p,q}$ in two non-compact parts. Let
$\Si_{p,q}$ be this (or these) dividing sphere.   

If the surface $\Si_{p,q}$
has two connected components then it bounds a domain
$\Omega_{p,q}$ diffeomorphic to $S^2 \times (a,b)$ (with
$a<b$). If $\Si_{p,q}$ has only one connected
components, then  $\Omega_{p,q}$  is empty, and can be viewed
as a limit case: $\Omega_{p,q}$ is diffeomorphic to 
$S^2 \times (a,b)$ with $a=b$ so that by extension, we say that 
$\partial \Omega_{p,q} = \Si_{p,q}$ even if 
$\Omega_{p,q}$ is empty. Finally, one can see that 
$$
\tilde{\Si}_{p,q} 
= 
\Si_{p,q} \cup \left( \cup_{i=1}^{p} \Si_i \right)
$$
where $p$ can be zero and where for all $i$, $\Si_i$ is a
$2$-sphere bounding a $3$-ball $\Om_i$. We get that 
\begin{equation} \label{ompq}
\tilde{\Om}_{p,q} = \Om_{p,q} \cup \left( \cup_{i=1}^{p} \Om_i \right),
\end{equation}
see Figure 1 for an illustration.

\begin{figure}
\centering\includegraphics{horizon-0.mps}
\caption{$\tilde{\Si}_{p,q} = \Si_{p,q}  \cup \Si_1 \cup \Si_2$,
$\tilde{\Om}_{p,q} = \Om_{p,q} \cup \Om_1 \cup \Om_2$} 
\end{figure}

\section{The $\Theta$-invariant, definition and basic properties} 
\label{definition_th}

Let $g$ be a metric on $S^3$ with positive Yamabe constant. We fix a 
point $p \in S^3$ and define $(S_{p,q},g_{p,q})$ as in Subsection
\ref{green}. Then, we define the $\Theta$-invariant by
$$
\tgp (q) 
\definedas
\frac{\left|\Om_{p,q}\right|_{g_{p,q}}}
{ \left| \Si_{p,q} \right|_{g_{p,q}}^{3/2}}
$$
for $q \in S^3 \setminus \{p \}$. The goal of this paper is to explore
properties of $\tgp$. We start with some basic properties. 

\begin{prop} \label{properties}
\begin{enumerate}
\item \label{prop_A}
The function $\Theta$ is conformally invariant. In other words, if $g$
and $g'$ are conformal then $\tgp(q)=\Theta^{g'}_p(q)$ for all 
$p,q \in S^3$, $p \neq q$.
\item \label{prop_B}  
The function $\Theta$ is symmetric in $p$ and $q$, that is 
$\tgp(q)= \Theta^g_q(p)$ for all $g$, $p,q \in S^3$, $p \neq q$.
\item \label{prop_C} 
If $\sigma$ stands for the round metric on $S^3$, then
$\Theta_p^{\sigma} \equiv 0$.
\end{enumerate}
\end{prop}


\begin{proof}
\ref{prop_A}: If $g'= u^4 g$ is a metric conformal to $g$ then 
the Green's functions $G_x'$ and $G_x$ for $L^{g'}$ and $L^{g}$ are
related by $G_p'=\frac{G_p}{u(p)u(\cdot)}$ for all $p \in S^3$. As a
consequence the metrics $g_{p,q}=(G_p + G_q)^4 g$ and 
$g_{p,q}' = {(G_p' + G_q')}^4 g'$ are proportional, more precisely we
have   
$$
g_{p,q}' =  \frac{(u(p) + u(q))^4}{u(p)^4u(q)^4} g_{p,q}.
$$ 
Let $\Si_{p,q}'$ be the outermost horizon bounding the domain
$\Om_{p,q}'$ in the metric $g_{p,q}'$. Then
$$ 
|\Si_{p,q}'|_{g_{p,q}'}  
= 
\frac{(u(p) + u(q))^4}{u(p)^4u(q)^4} |\Si_{p,q}|_{g_{p,q}} 
$$
and
$$
|\Om_{p,q}'|_{g_{p,q}'} 
= 
\frac{(u(p) + u(q))^6}{u(p)^6 u(q)^6}
|\Om_{p,q}|_{g_{p,q}}.
$$
Here the notation $|\cdot|_h$ indicates that the area (or volume) is
computed using the metric $h$. We get 
$$
\frac{|\Om_{p,q}'|_{g_{p,q}'}}
{|\Si_{p,q}'|_{g_{p,q}'}^{3/2}}
=
\frac{|\Om_{p,q}|_{g_{p,q}}}
{|\Si_{p,q}|_{g_{p,q}}^{3/2}},
$$
and hence $\Theta^{g'}_p(q)= \tgp(q)$ which proves
Property \ref{prop_A}.  

\ref{prop_B}: Obvious from the definition of $\tgp$.  

\ref{prop_C}: Let $p,q \in S^3$ be fixed. Clearly, there exists a
conformal diffeomorphism $\alpha$ of $(S^3, \sigma)$ with
$\alpha(p) = p$ and $\alpha(q) = -p$. By Property
\ref{prop_A}, we can then assume that $q = -p$. From Subsection
\ref{green} we see that $(S_{p,q},g_{p,q})$ is isometric to 
$\mR^3 \setminus \{0 \}$ equipped with the Schwarzschild metric. In
particular, $\Omega_{p,q}$ is empty from which Property \ref{prop_C}
follows.
\end{proof}

\section{An upper bound for $\Theta$} \label{upper_bound}

In this section we prove the following result.

\begin{theorem} \label{theta_leq}
For all $p,q \in S^3$, $p \neq q$, and all metrics $g$ such that 
$\mu(g) > 0$ we have 
$$
\mu(g) \left( 1 + \frac{4}{\sqrt{\pi}} \tgp(q)\right)^{1/3} 
\leq 
\mu(\sigma)
$$
where $\sigma$ is the standard round metric on $S^3$. 
\end{theorem}

\begin{cor}
The function $\tgp$ is bounded on $S^3 \setminus \{ p\}$. 
\end{cor}

Note that Theorem \ref{theta_leq} provides an alternative proof of
Property \ref{prop_C} in Proposition \ref{properties}.

In our mind, the main interest of this result, as well as Theorem
\ref{yam_theta} in next Section, is to exhibit how the
$\Theta$-invariant is closely related with the Yamabe constant. Its
proof relies on convexity inequalities combined with Bray and Neves
techniques \cite{bray_neves_04} using the weak inverse mean curvature
flow. The trick here is to apply these techniques on 
$( S_{p,q} \setminus \overline{\Om_{p,q}}, g_{p,q})$ which consists in
two connected components, each of them being an asymptotically flat
manifold.

We begin with a technical lemma.

\begin{lemma}
Let $(N,h)$ be an asymptotically flat manifold whose boundary is the
outermost compact minimal surface. Let $(\sch,g_{\sch})$ be one half of
the spatial Schwarzschild manifold with $m=2$, see 
(\ref{schwarzschild_def}),
whose boundary is the minimal sphere $\{ |x| = 1 \}$. 
Let also $\Phi$ and $\Phi^{\sch}$  be the functions given by Theorem
\ref{mcf} and associated to the weak inverse mean curvature flow on
$(N,h)$ and $(\sch,g_{\sch})$.
 
Finally, denote by $\Si_t\definedas \{\Phi=t\} $ and
$\Si^{\sch}_t \definedas \{ \Phi^{\sch}=t \}$ the level sets of $\Phi$ and
$\Phi^{\sch}$. Then, 
\begin{equation} \label{bnlemma1}
\int_{\Si_t} |\nabla \Phi|_h \, da^h 
\leq 
\sqrt{\frac{|\Si_0|_h}{ |\Si_0^{\sch}|_{g_{\sch}} }}
\int_{\Si_t^{\sch}} |\nabla \Phi^{\sch}|_{g_{\sch}} \, da^{g_{\sch}}
\end{equation}
and 
\begin{equation} \label{bnlemma2}
\int_{\Si_t} \frac{1}{ |\nabla \Phi|_h } \, da^h 
\geq
\left( \frac{|\Si_0|_h }{ |\Si_0^{\sch}|_{g_{\sch}} } \right)^{ 3/2 } 
\int_{\Si_t^{\sch}} \frac{1}{|\nabla \Phi^{\sch}|_{g_{\sch}}} \, da^{g_{\sch}}  
\end{equation}
for almost all $t$. Here if $\Si$ is a compact surface, $|\Si|_h$
denotes its area in the metric $h$. 
\end{lemma} 

The proof of this lemma is entirely contained in the work of Bray and
Neves \cite{bray_neves_04}, but not stated in this way. So, we recall
the proof here. The integrals above are not obviously convergent since  
$\nabla \Phi$ can have zeros, the existence of the integrals is
carefully justified in \cite{bray_neves_04} and we do not recall all
details of these arguments.
\begin{proof}
Let 
$$
m_H(\Si_t)
=
\sqrt{\frac{|\Si_t|_h}{16 \pi}}
\left( 16 \pi - \int_{\Si_t} |\nabla \Phi|_h^2 \, da^h \right)
$$
be the Hawking mass of $\Si_t$. By Theorem \ref{mcf} we have
$$
m_H(\Si_t) \geq m_H(\Si_0) = \sqrt{\frac{|\Si_0|_h}{16 \pi}}.
$$
This gives
$$
\int_{\Si_t} |\nabla \Phi|_h^2 \, da^h 
\leq 
16 \pi \left(1 - e^{ -t/2 } \right)
$$
since $|\Si_t|_h = |\Si_0|_h e^t$. By the Cauchy-Schwarz inequality, 
\begin{equation*}
\begin{split} 
\int_{\Si_t} |\nabla \Phi|_h \, da^h 
&\leq 
\sqrt{|\Si_t|_h} \left(\int_{\Si_t} |\nabla \Phi|_h^2 \, da^h
\right)^{1/2} \\
&\leq 
\sqrt{|\Si_0|_h} e^{ t/2 } \sqrt{16 \pi(1 - e^{ -t/2 }) } 
\end{split}
\end{equation*}
and finally
\begin{equation} \label{nablaphi}
\int_{\Si_t} |\nabla \Phi|_h \, da^h 
\leq 
\sqrt{16 \pi | \Si_0|_h (e^t - e^{ t/2 }) }.
\end{equation} 
Observe that the Hawking mass is constant for the inverse mean
curvature flow on $(\sch,g_{\sch})$ and also that 
$|\nabla \Phi^{\sch}|_{g_{\sch}}$ is constant on the corresponding
$\Si_t^{\sch}$. So the above reasoning is still valid on
$(\sch,g_{\sch})$ but all inequalities become equalities. In other
words we have
$$
\int_{\Si_t^{\sch}} |\nabla \Phi^{\sch}|_{g_{\sch}} \, da^{g_{\sch}} 
= 
\sqrt{16 \pi |\Si_0^{\sch}|_{g_{\sch}} (e^t - e^{ t/2 }) }.
$$
Together with Inequality \eref{nablaphi} we get Inequality
\eref{bnlemma1}.

The H\"older inequality tells us that 
$$
\int_{\Si_t} \frac{1}{ |\nabla \Phi|_h} \, da^h  
\geq 
\frac{|\Si_t|^2}{\int_{\Si_t} |\nabla \Phi|_h \, da^h },
$$
and
$$
\int_{\Si_t^{\sch}} \frac{1}{ |\nabla \Phi^{\sch}|_{g_{\sch}}} \, da^{g_{\sch}} 
= 
\frac{|\Si^{\sch}_t|_{g_{\sch}}^2}
{\int_{\Si^{\sch}_t} |\nabla \Phi^{\sch}|_{g_{\sch}} \, da^{g_{\sch}} }
$$
since $|\nabla \Phi^{\sch}|_{g_{\sch}}$ is constant on $\Si_t^{\sch}$. Using
this together with the observation  
$$
\frac{|\Si_t|_h}{|\Si_t^{\sch}|_h}
= 
\frac{|\Si_0|_{g_{\sch}}}{|\Si_0^{\sch}|_{g_{\sch}}}
$$
we conclude that \eref{bnlemma2} holds.
\end{proof}

With this Lemma we are prepared to prove the Theorem. 

\begin{proof}[Proof of Theorem \ref{theta_leq}]
Using the same notation as in Equation \eref{ompq}, we see that the set
$S_{p,q} \setminus \overline{\Om_{p,q}}$ has exactly two
connected components diffeomorphic to $\mR^3 \setminus B$ (where $B$
is the unit ball of $\mR^3$ and whose boundary is an outermost horizon. 
We denote by $M_1$ and $M_2$ these two connected components (which
are asymptotically flat manifolds in the metric $g_{p,q}$). Let
$\Si_1$ and $\Si_2$ stand for the respective boundaries of $M_1$ and
$M_2$. Note that if $\Om_{p,q} = \emptyset$ then $\Si_1=\Si_2=\Si_{p,q}$ 
(with the notations of the end of Paragraph \ref{definition_th} ) and
if $\Om_{p,q} \neq \emptyset$ then $\Si_1$ and $\Si_2$ are exactly the
two connected components of $\Si_{p,q}$.  
By Theorem \ref{mcf} there are functions $\Phi_1$ and $\Phi_2$ on
$M_1$ and $M_2$ defining the weak inverse mean curvature flow. Denote
the level sets of $\Phi_1$ and $\Phi_2$ by $\Si_t^1$ and $\Si_t^2$. 
Define a test function on $S_{p,q}$ by 
\begin{equation*}
u
\definedas      
\begin{cases}
f \circ \Phi_1  &\text{on $M_1$,} \\
f \circ \Phi_2  &\text{on $M_2$,} \\
1               &\text{on $\Om_{p,q}$,}
\end{cases}
\end{equation*}
where 
$$ 
f(t) \definedas \frac{1}{\sqrt{2e^t - e^{ t/2 }}}.
$$
Let $w \in C^{\infty}(M)$, $w \neq 0$. Let also for $\ep$ small
$\eta_\ep \in C^\infty(M)$, $\eta_\ep \in [0,1]$ be a cut-off function
such that $\eta_\ep \equiv 0$ on $B^g_p(\ep) \cup B^g_q(\ep) )$,
$\eta_\ep \equiv 1$ on $S^3 \setminus (B^g_p(2 \ep) \cup B^g_q(\ep))$,
and $|\nabla \eta_\ep|_g \leq 2/\ep$. One easily computes that 
$$ 
\lim_{\ep \to 0} 
\frac
{\int_{S^3} (8 |\nabla (\eta_\ep w)|_g^2 + \Scal^g (\eta_\ep w)^2) \, dv^g} 
{\left( \int_{S^3} (\eta_\ep w)^6 \, dv^g \right)^{ 1/3 } }
= 
\frac
{\int_{S^3} (8 |\nabla  w|_g^2 + \Scal^g w^2) \, dv^g}
{\left( \int_{S^3} w^6 \, dv^g \right)^{ 1/3 } }.
$$
By the definition of $\mu(g)$, we then get that
$$
\mu(g) 
= 
\inf \frac{\int_{S^3} (8 |\nabla  w|_g^2 + \Scal^g  w^2) \, dv^g}
{\left( \int_{S^3} w^6 \, dv^g \right)^{ 1/3 } }
$$
where the infimum is taken over all smooth non-zero functions 
$w$ which are identically zero in a neighborhood of $p$ and $q$. Since
$\mu$ is conformally invariant and since $C^\infty$ is dense in $C^{\sch}$
we have  
$$
\mu(g) 
= 
\inf
\frac{8 \int_{S_{p,q}} |\nabla  w|_{g_{p,q}}^2 \, dv^{g_{p,q}} }
{\left( \int_{S_{p,q}} w^6 \, dv^{g_{p,q}} \right)^{ 1/3 } }.
$$
where the infimum is taken over all non-zero functions locally
Lipschitz functions $w \in H_1^2(S_{p,q},g_{p,q})$. Here 
$H_1^2(S_{p,q},g_{p,q})$ denotes the set of functions in $L^2$ whose
derivatives are also in $L^2$. It follows from \cite{bray_neves_04}
that $u \in H_1^2(S_{p,q},g_{p,q})$ is locally Lipschitz. This implies  
\begin{equation} \label{y(g)}
\mu(g) 
\leq  
\frac{8 \int_{S_{p,q}} |\nabla u|_{g_{p,q}}^2 \, dv^{g_{p,q}} }
{\left( \int_{S_{p,q}} u^6 \, dv^{g_{p,q}} \right)^{ 1/3 } }.
\end{equation}

We have, 
\begin{equation*} 
\int_{S_{p,q}} |\nabla u|_{g_{p,q}}^2 \, dv^{g_{p,q}}
=  
\int_{M_1} f'(\Phi_1)^2 |\nabla \Phi_1|_{g_{p,q}}^2 \, dv^{g_{p,q}} 
+ 
\int_{M_2} f'(\Phi_2)^2 |\nabla \Phi_2|_{g_{p,q}}^2 \, dv^{g_{p,q}}.
\end{equation*}
Let $(\sch, g_{\sch})$ be one half of the spatial Schwarzschild manifold
and let $\Phi^{\sch}$ be the function associated to the weak inverse mean
curvature flow on $(\sch,g_{\sch})$ and whose level sets will be denoted
by $\Si_t^{\sch}$.  We set $a_0 \definedas|\Si_0^{\sch}|_{g_{\sch}}$
and $a_i \definedas |\Si_0^i|_{g_{p,q}}$ for $i = 1,2$. One can compute that 
\begin{equation} \label{nablavalue}
|\nabla \Phi^{\sch}|_{g_{\sch}}
\equiv 
\sqrt{\frac{16 \pi}{a_0}} \frac{\sqrt{e^t - e^{ t/2 }}}{e^t}
\end{equation}
on $\Si^{\sch}_t$. By the coarea formula, Inequality \eref{bnlemma1}, and
the fact that $|\Si_t^{\sch}|_{g_{\sch}} = a_0 e^t$ we have 
\begin{equation*}
\begin{split}
\int_{M_1} f'(\Phi_1)^2 |\nabla \Phi_1|_{g_{p,q}}^2 \, dv^{g_{p,q}} 
&=
\int_0^{\infty} f'(t)^2 \left(
\int_{\Si_t^1} |\nabla \Phi_1|_{g_{p,q}} \, da^{g_{p,q}}
\right) \, dt\\
&\leq 
\sqrt{\frac{a_1}{a_0}} \int_0^{\infty} f'(t)^2 \left(
\int_{\Si_t^{\sch}} |\nabla \Phi^{\sch}|_{g_{\sch}} \, da^{g_{\sch}} 
\right) \, dt \\
&= 
\sqrt{16 \pi} \sqrt{a_1} I
\end{split}
\end{equation*}
where $I$ is defined as
$$
I \definedas
\int_0^{\infty} f'(t)^2 \sqrt{e^t - e^{ t/2 }} \, dt.
$$
Doing the same on $M_2$ and inserting the value of $I$ which is
computed in Lemma \ref{integral} in Appendix \ref{evalint} we obtain 
\begin{equation} \label{num}
\int_{S_{p,q}} |\nabla u|_{g_{p,q}}^2 \, dv^{g_{p,q}}
= 
\frac{3 \pi^{3/2}}{8} \left( \sqrt{a_1} + \sqrt{a_2} \right).
\end{equation}
Now we write 
$$
\int_{S_{p,q}} u^6 \, dv^{g_{p,q}}
= 
\int_{M_1} f(\Phi_1)^6 \, dv^{g_{p,q}}
+
\int_{M_2} f(\Phi_2)^6 \, dv^{g_{p,q}} 
+ 
|\Om_{p,q}|_{g_{p,q}}.
$$
By the coarea formula, Inequality \eref{bnlemma2}, and Equation
\eref{nablavalue} we get
\begin{equation*}
\begin{split}
\int_{M_1} f(\Phi_1)^6 \, dv^{g_{p,q}} 
&= 
\int_0^{\infty} f(t)^6 
\left(
\int_{\Si^1_t} \frac{1}{|\nabla \Phi_1|_{g_{p,q}}} \, da^{g_{p,q}} 
\right) \, dt \\ 
&\geq 
\left(\frac{a_1}{a_0} \right)^{3/2}
\int_0^{\infty} f(t)^6 \left( 
\int_{\Si^{\sch}_t} \frac{1}{|\nabla \Phi^{\sch}|_{g_{\sch}}} \, da^{g_{\sch}} 
\right) \, dt
\\
&= 
a_1^{3/2} \frac{1}{\sqrt{16\pi}} J
\end{split}
\end{equation*}
where $J$ is defined as
$$
J \definedas
\int_0^{\infty} \frac{f(t)^6 e^{2t}}{(e^t - e^{ t/2 })} \, dt.
$$
Hence, doing the same on $M_2$ and using the value of $J$ from Lemma
\ref{integral} we obtain 
\begin{equation} \label{den}
\int_{S_{p,q}} u^6 \, dv^{g_{p,q}} 
\geq 
\frac{\sqrt{\pi}}{8} ( a_1^{3/2} + a_2^{3/2} ) + |\Om_{p,q}|_{g_{p,q}}.
\end{equation}
Plugging \eref{num} and \eref{den} in \eref{y(g)}, we conclude
\begin{equation*}
\mu(g) 
\leq
8 \frac{\frac{3 \pi^{3/2}}{8} \left( \sqrt{a_1} + \sqrt{a_2} \right)} 
{\left( 
\frac{\sqrt{\pi}}{8} ( a_1^{3/2} + a_2^{3/2} ) + |\Om_{p,q}|_{g_{p,q}}
\right)^{1/3} }.
\end{equation*} 
It follows that 
\begin{equation} \label{yamg}
\mu(g) 
\leq 
\frac{3 \pi^{3/2} A_1 } 
{\left( 
\frac{\sqrt{\pi}}{8} A_2 + \frac{|\Om_{p,q}|_{g_{p,q}}}{(a_1 + a_2)^{ 3/2 }}
\right)^{1/3} }
\end{equation} 
where 
$$
A_1 \definedas \frac{\sqrt{a_1}+\sqrt{a_2}}{\sqrt{a_1 + a_2}}
\quad \text{and} \quad 
A_2 \definedas \frac{ a_1^{3/2} + a_2^{3/2} }{ (a_1 + a_2)^{3/2} }.
$$ 
Elementary arguments show that 
\begin{equation} \label{A12}
A_1 \leq \sqrt2 
\quad \text{and} \quad
A_2 \geq \frac{1}{\sqrt2}.
\end{equation}
%

Note that if $\Om_{p,q} \neq \emptyset$, then  
$a_1+a_2 \leq |\Si_{p,q}|_{g_{p,q}}$. If $\Om_{p,q} = \emptyset$, then
the boundaries of $M_1$ and $M_2$ are exactly
$\Si_{p,q}$ which is a connected component of $\Si_{p,q}$ and
hence $a_1 + a_2 =2  |\Si_{p,q}|_{g_{p,q}}$. In both cases,
$a_1 + a_2 \leq 2 |\Si_{p,q}|_{g_{p,q}}$ and 
\begin{equation} \label{si}
\frac{|\Om_{p,q}|_{g_{p,q}}}{(a_1+a_2)^{ 3/2 }} 
\geq 
\frac{\tgp(q)}{2^{ 3/2 }}
\end{equation}
Plugging \eref{A12} and \eref{si} in \eref{yamg} we get
$$
\mu(g) 
\leq
\frac{3 \pi^{3/2} \sqrt2 } 
{\left( 
\frac{\sqrt{\pi}}{8} \frac{1}{\sqrt2}
+ 
\frac{\tgp(q)}{2^{ 3/2 }}
\right)^{1/3} } \\
=
\frac{6 (2 \pi^2)^{ 2/3 } } 
{\left( 1 + \frac{4}{\sqrt{\pi}} \tgp(q) \right)^{1/3} }, 
$$
and Theorem \ref{theta_leq} follows since 
$\mu(\sigma) = 6 (2 \pi^2)^{ 2/3 }$.  
\end{proof}


\section{Metrics with large $\Theta$-invariant} 
\label{metrics_large}

Note that the upper bound of $\tgp(q)$ provided by Theorem
\ref{theta_leq} will tend to infinity if the metric $g$ tends to a
metric with vanishing Yamabe constant. In the following theorem we 
prove that $\tgp(q)$ itself tends to infinity in that situation.

\begin{theorem} \label{yam_theta}
Let $g^\infty$ be a Riemannian metric on $S^3$ and assume that $(g^k)$
is a sequence of Riemannian metrics with positive Yamabe constant
tending to $g^\infty$ in all $C^l$, $l \in \mN$ as $k \to \infty$.  
Let $p,q \in S^3$, $p \neq q$. Then 
$$
\mu(g^\infty) = 0 
\Leftrightarrow 
\lim_{k \to \infty} \Theta_p^{g^k}(q) = \infty.
$$
\end{theorem}

This theorem implies in particular that if the metric $g$ is close
enough to a metric of zero Yamabe constant then 
$\tgp \not\equiv 0$. The proof is inspired by an argument of Beig and 
O'Murchadha \cite{beig_omurchadha_91}. The main result of
\cite{beig_omurchadha_91} is that 
$(S_p,g^k_p) \definedas (S^3 \setminus \{ p \}, (G^k_p)^4 g^k)$ 
contains a trapped compact minimal surface if $\mu(g^\infty) = 0$ and 
$k$ is large enough.  

\begin{proof}
First, Theorem \ref{theta_leq} tells us that 
$$ 
\lim_{k \to \infty} \Theta_p^{g^k}(q) = \infty  
\Rightarrow \mu(g^\infty) = 0.
$$
To show the opposite implication we assume that 
$\mu(g^\infty) = 0$. Since $\mu$ and $\tgp$ are conformally invariant,
we can further assume that  
$\Scal^{g^\infty} \equiv 0$ and therefore there exists a sequence
$(\ep_k)$ tending to $0$ such that 
\begin{equation} \label{scal}
\| \Scal^{g^k} \|_{L^\infty} \leq \ep_k.
\end{equation}
Let $\eta \in C^\infty( [0,\infty) )$ be a cut-off function
satisfying $0 \leq \eta \leq 1$, $\eta \equiv 1$ on $[0,\de)$,
$\eta \equiv 0$ on $[2 \de, \infty)$, $\de$ being a fixed small
number. Denote by $G_p^k$ the Green's function for $L^{g^k}$ at a
point $p \in S^3$. Then, by \eref{expan}, there exists a function 
$\alpha_p^k \in C^{\infty}(S^3)$ such that 
$$
G_p^k 
= 
\frac{\eta(r_k)}{r_k} + \alpha_p^k
$$
where we use the notation $r_k = d^{g^k}(p,\cdot)$.  
We start by proving the following result.

\begin{lemma} \label{green_f}
There is a subsequence of $(g^k)$ for which the corresponding
functions $\alpha_p^k$ can be decomposed as 
$$
\alpha_p^k = a_p^k + \beta_p^k
$$
where $(a_p^k)$ is a sequence of real numbers tending to $\infty$ and
where $\beta_p^k \in C^{\infty}$ is a smooth function such that 
$$ 
\int_{S^3} \beta_p^k \, dv^{g^k} = 0
$$
and such that 
\begin{equation} \label{betak}
\|\beta_p^k \|_{C^1}  = o(a_p^k).
\end{equation}
\end{lemma} 

Here the notation in the last claim means that 
$\|\beta_p^k \|_{C^1} / a_p^k$ tends to zero as $k \to \infty$. In the
following proof $C>0$ stands for a constant which is independent of
$k$ but may change from line to line.

\begin{proof}[Proof of Lemma \ref{green_f}]
First we prove that 
\begin{equation} \label{first_step}
\limsup_{k \to \infty} \int_{S^3} \alpha_p^k \, dv^{g^k} 
= 
\infty.
\end{equation}
From the definition of $G_p^k$ together with \eref{expan} we have
\begin{equation*} 
\begin{split}
1 
&= 
\int_{S^3} G_p^k (L^{g^k} 1) \, dv^{g^k} \\
&\leq 
\| L^{g^k} 1\|_{L^\infty} \int_{S^3} G_p^k \, dv^{g^k} \\
&\leq 
\| \Scal^{g^k} \|_{L^\infty} 
\left( 
\int_{B^{g^k}_p(2 \de)} r_k^{-1} \, dv^{g^k} 
+ 
\int_{S^3} \alpha_p^k \, dv^{g^k} 
\right) .
\end{split}
\end{equation*}
By Lebesgue's Theorem 
$$
\lim_{k \to \infty} \int_{B^{g^k}_p(2 \de)} r_k^{-1} \, dv^{g^k} 
= 
\int_{B^{g^\infty}_p(2 \de)} r_\infty^{-1} \, dv^{g^\infty} 
<
\infty
$$
where $r_\infty = d^{g^\infty}(p,\cdot)$. To get a contradiction we
assume that (\ref{first_step}) does not hold. Then 
$\int_{S^3} \alpha_p^k \, dv^{g^k}$ is bounded as $k$ goes to $\infty$. 
We conclude that 
$$
1 \leq C  \| \Scal^{g^k} \|_{L^\infty}.
$$
By Equation \eref{scal}, the right-hand side of this inequality tends
to zero which is not possible. This proves (\ref{first_step}).

Next we set 
$$
a_p^k 
\definedas 
\int_{S^3} \alpha_p^k \, dv^{g^k} 
\bigg/
\int_{S^3} \, dv^{g^k}  
$$
and $\beta_p^k \definedas \alpha_p^k - a_p^k$ so that 
$$
\int_{S^3} \beta_p^k \, dv^{g^k} = 0.
$$
Since by assumption $g^k \to g^\infty$ in all $C^l$, $l \in \mN$,
the Sobolev inequality
\begin{equation} \label{C1_est}
\| \beta_p^k \|_{C^1} \leq  C \| \beta_p^k \|_{H_2^4} 
\end{equation}
holds with a constant $C$ independent of $k$. Here, for any $s>1$, 
$H_2^s$ is the space of $L^s$ functions whose derivatives of
first and second order belong to $L^s$. Let now $l >1$. By standard
regularity result, see for example \cite[Theorem 2.4]{lee_parker_87},   
\begin{equation} \label{h2l}
\| \beta_p^k \|_{H_2^l} 
\leq 
C \left( 
\| \Delta^{g^k} \beta_p^k \|_{L^l} 
+ \| \beta_p^k \|_{L^l} 
\right).
\end{equation}
Using \eref{scal} we write
\begin{equation*}
\begin{split} 
\| 8 \Delta^{g^k} \beta_p^k \|_{L^l}
&= 
\| L^{g^k} \beta_p^k - \Scal^{g^k} \beta_p^k \|_{L^l}\\
&\leq 
\left( \|L^{g^k} \beta_p^k \|_{L^l} + \ep_k \|\beta_p^k \|_{L^l} \right).
\end{split}
\end{equation*}
From the definition of $\beta_p^k$ together with \eref{scal} we get
\begin{equation*}
\begin{split}
\| L^{g^k} \beta_p^k \|_{L^l} 
&=  
\left\| L^{g^k} G_p^k - L^{g^k} \left(\frac{\eta(r_k)}{r_k} \right)
- L^{g^k} a_p^k \right\|_{L^l}  \\
&\leq 
\left\| \delta_p - \eta(r_k) \Delta^{g^k} r_k^{-1} \right\|_{L^l} 
+ \left\| \frac{\Delta^{g^k} \eta(r_k)}{r_k} \right\|_{L^l} \\
&\quad 
+ 2 \left\| 
g^k(\nabla^{g^k} \eta(r_k), \nabla^{g^k} r_k^{-1} ) 
\right\|_{L^l} 
+ C \ep_k \| \beta_p^k \|_{L^l} +C \ep_k a_p^k.
\end{split}
\end{equation*}
Since the derivatives of $\eta(r_k)$  are  supported in $M
\setminus B^{g^k}_p(\de)$, the second and third terms of the right hand
side in the expression above are bounded by some constant $C>0$
independent of $k$. One can also compute 
\cite[Section 6]{lee_parker_87}
$$
\| \delta_p -\eta(r_k) \Delta^{g^k} r_k^{-1}\|_{L^\infty} \leq C.
$$
Finally, we obtain 
\begin{equation} \label{ll}
\| \Delta^{g^k} \beta_p^k \|_{L^l}  
\leq 
C (1 + \ep_k \|\beta_p^k\|_{H_l} + \ep_k a_p^k).
\end{equation}
In particular, we easily deduce from \eref{h2l} that 
\begin{equation} \label{h2l1}
\| \beta_p^k \|_{H_2^l} 
\leq 
C \left( 
( 1  + \ep_k a_p^k)
+ \| \beta_p^k \|_{L^l} 
\right).
\end{equation}
Let $\la_k$ denote the first eigenvalue of $\Delta^{g^k}$. Since
$\int_M \beta_p^k dv^{g^k} =0$ the Cauchy-Schwarz inequality tells us that 
\begin{equation*}
\begin{split}
\int_M (\beta_p^k)^2 dv^{g^k}  
&\leq 
\frac{1}{\la_k} \int_M |\nabla \beta_p^k |_{g^k}^2 dv^{g^k} \\
&\leq 
\frac{1}{\la_k} \int_M \beta_p^k \Delta^{g^k} \beta_p^k dv^{g^k} \\
&\leq 
\frac{1}{\la_k} 
\left( \int_M (\beta_p^k)^2 dv^{g^k} \right)^{1/2} 
\left( \int_M (\Delta^{g^k} \beta_p^k)^2 dv^{g^k} \right)^{1/2}. 
\end{split}
\end{equation*}
The sequence $(\la_k)$ has a non-zero limit since the metrics $(g^k)$
converges, and hence the sequence $(\la_k^{-1})$ is bounded. Together
with \eref{ll} applied with $l=2$, one gets 
$$
\|\beta_p^k \|_{L^2} \leq C(1 + \ep_k a_k).
$$
Returning to \eref{h2l1}, we obtain 
$$
\|\beta_p^k \|_{H^2_2} \leq C(1 + \ep_k a_k).
$$
In particular, by the Sobolev embedding theorem, we get that 
$$
\|\beta_p^k \|_{L^4} \leq C(1 + \ep_k a_k).
$$
Setting $l=4$ and inserting this inequality in \eref{h2l1}, we get 
$$
\|\beta_p^k \|_{H_2^4} \leq C(1 + \ep_k a_k).
$$
Together with \eref{C1_est} this ends the proof of 
Lemma \ref{green_f}.
\end{proof}
 
Let us return to the proof of Theorem \ref{yam_theta}. We fix
points $p,q \in S^3$, $p \neq q$, and to get a contradiction we
assume that $\Theta_p^{g^k}(q)$ has a bounded subsequence. 
Define $G^k \definedas G_p^k + G_q^k$ so that $g^k_{p,q}= G_k^4 g^k$. 
For $r>0$ small let $S_p^k(r)$ be the sphere defined by $r_k = r$,
where again $r_k = d^{g^k}(p,\cdot)$.  
Using the transformation formula for mean curvature under a conformal
change of the metric (see for example \cite[Equation 1.4]{escobar_92})
one can compute that the mean curvature of $S_p^k(r)$ in the metric
$g^k_{p,q}$ is
$$
H_k 
= 
\frac{1}{G_k^3} 
\left(2 g^k(\nabla^{g^k} r_k, \nabla^{g^k} G_k) 
+ \left( \frac{1}{r} + O(r) \right) G_k \right) 
$$
where the constant involved in the bound of the ordo term is
independent of $k$. Apply Lemma \ref{green_f} to $G_p^k$ and $G_q^k$
and take further subsequences to get the corresponding $a_p^k$ and
$a_q^k$ which tend to infinity. Set $a_k \definedas a_p^k + a_q^k$, by
\eref{betak} we have 
$$
G_k
= 
r_k^{-1} + a_k + o(a_k)
$$ 
near $p$. It follows that the mean curvature of $S_p^k(r)$ satisfies
$$ 
H_k
= 
\frac{1}{G_k^3 r^2} 
\left( - 1 + (a_k+o(a_k)) r + r^2  o(a_k)  \right),
$$
see Appendix \ref{meancurv} for further details. 
In particular, the sphere $S_p^k (2/a_k)$ has positive mean curvature
whereas the sphere $S_p^k(1/2a_k)$ has negative mean curvature when
$k$ is large. By standard existence results, there exists a minimal
$2$-sphere $\Si_p^k$ lying between these two spheres. Doing the same
near $q$, we get the existence of a minimal $2$-sphere
$\Si_q^k$. Clearly, 
$\Si_{p,q} = \tilde{\Si}_{p,q} = \Si_p^k \cup \Si_q^k$.  We also get 
\begin{equation} \label{ompqgeq}
S^3 \setminus 
\left( D^{g^k}_p(2/a_k ) \cup B^{g^k}_q( 2/a_k ) \right) 
\subset \Om_{p,q}. 
\end{equation} 
Since $\Si_p^k$ is minimal and since $G_k \leq  C a_k$ on 
$S^k_p (1/2)$ the area of $\Si_p^k$ satisfies
\begin{equation*}
\begin{split} 
| \Si_p^k |_{g^k_{p,q}} 
&\leq 
| S^k_p ( 1/2a_k ) |_{g^k_{p,q}} \\ 
&= 
\int_{S^k_p ( 1/2a_k ) } \, da^{g^k_{p,q}} \\
&= 
\int_{S^k_p ( 1/2a_k ) } G_k^{4} \, da^{g^k} \\
&\leq 
C a_k^4 \int_{S^k_p ( 1/2a_k ) } \, da^{g^k_{p,q}} \\ 
&\leq C a_k^2.
\end{split}
\end{equation*}
Doing the same for $\Si_q^k$ we get that 
\begin{equation} \label{thtoinfty}
| \Si_{p,q} |_{g^k_{p,q}} \leq C a_k^2.
\end{equation}
Using \eref{ompqgeq}, we have
\begin{equation*} 
\begin{split}
| \Om_{p,q} |_{g^k_{p,q}} 
&\geq 
\int_{S^3 \setminus ( B^{g^k}_p( 2/a_k ) \cup B^{g^k}_q( 2/a_k )) }
\, dv^{g^k_{p,q}} \\
&\geq 
\int_{S^3 \setminus (B^{g^k}_p( 2/a_k ) \cup B^{g^k}_q( 2/a_k )) }
G_k^6 \, dv^g. 
\end{split}
\end{equation*}
Estimate \eref{betak} implies that $G_k \geq C a_k$ on 
$S^3 \setminus (B^{g^k}_p ( 2/a_k ) \cup B^{g^k}_q ( 2/a_k ))$. 
This leads to 
$$
| \Om_{p,q} |_{g^k_{p,q}} 
\geq 
C a_k^6 
\int_{S^3 \setminus (B^{g^k}_p( 2/a_k ) \cup B^{g^k}_q( 2/a_k ))} \, dv^{g^k}
\geq 
C a_k^6.
$$
Together with \eref{thtoinfty}, we get
$$
\Theta^{g^k}_p(q) 
= 
\frac{|\Om_{p,q}|_{g^k_{p,q}}}{|\Si_{p,q}|_{g^k_{p,q}}^{3/2} } 
\geq 
C a_k^3
$$
and hence $\Theta^{g^k}_p(q)$ for the subsequence cannot be
bounded. This proves Theorem \ref{yam_theta}.
\end{proof}

\appendix
\section{Evaluation of integrals} 
\label{evalint}

Here we indicate how to evaluate two definite integrals needed in the
proof of Theorem \ref{theta_leq}. Compare the discussion in
\cite{bray_neves_04}, pages 421-422. 

\begin{lemma} \label{integral}
Let
$$ 
f(t) \definedas \frac{1}{\sqrt{2e^t - e^{t/2}}}.
$$
for $t \in (0, \infty)$ and set 
$$
I 
\definedas
\int_0^{\infty}  f'(t)^2 \sqrt{e^t - e^{ t/2 }} \, dt 
\quad \text{and} \quad 
J
\definedas
\int_0^{\infty} \frac{f(t)^6 e^{2t}}{(e^t - e^{ t/2 })} \, dt.
$$
Then 
$$
I = \frac{3 \pi }{32} 
\quad \text{and} \quad 
J = \frac{\pi}{2}.
$$
\end{lemma} 

\begin{proof}
Observe that 
$$
I
= 
\frac{1}{16} \int_0^{\infty} 
\frac{(4 e^{t/2} - 1)^2}{(2e^{ t/2 } - 1)^3} 
\sqrt{\frac{e^{ t/2 }-1}{e^{ t/2 } } } \, dt.
$$
Through the change of variables 
$s = \sqrt{\frac{e^{ t/2 }-1}{e^{ t/2 }} }$,
that is $e^{ t/2 } = \frac{1}{1-s^2}$ and 
$dt = \frac{4sds}{1-s^2}$, we get 
$$
I
= 
\frac{1}{4} \int_0^1 \frac{(3+ s^2)^2s^2}{(s^2 +1)^3} \, ds.
$$
Writing 
$$
\frac{(3 +s^2)^2s^2}{(s^2 +1)^3} 
=
1 +  \frac{3}{s^2 +1} - \frac{4}{(s^2 +1)^3},
$$ 
one gets
$$
I
= 
\frac{1}{4} + \frac{3}{4} \left[ \arctan t\right]_0^1 
- \left[ \frac38 \arctan t + \frac38 \frac{t}{1 + t^2} 
+ \frac14 \frac{t^2}{1+t^2}  \right]_1^{\infty} = \frac{3\pi}{32}.
$$

In the same way, observe that 
$$
J = 
\int_0^{\infty} \frac{1}{(2e^{ t/2 } - 1)^2} 
\sqrt{ \frac{e^{ t/2 }}{e^{ t/2 } - 1} } \, dt.
$$
Using the change of variables 
$s = \sqrt{\frac{e^{ t/2 }}{e^{ t/2 } - 1} }$, that is 
$e^{ t/2 } = \frac{s^2}{s^2 -1}$ and 
$dt = -\frac{4ds}{s(s^2 -1)}$, we get
$$
J = 
\int_1^{\infty} \frac{4(s^2 -1)^2}{(s^2 +1)^3} \, ds.
$$
From 
$$
\frac{4(s^2 -1)^2}{(s^2 +1)^3} 
= 
\frac{4}{s^2 -1} - \frac{16}{(s^2 - 1)^2} +\frac{16}{(s^2 -1)^3},
$$ 
we have
\begin{equation*} 
\begin{split}
J 
&= 
4 \left[ \arctan t\right]_1^{\infty} 
- 16 \left[ \frac12 \arctan t +
 \frac12 \frac{t}{1 + t^2} \right]_1^{\infty}\\
&\quad
+ 16 \left[ \frac38 \arctan t + \frac38 \frac{t}{1 + t^2}  
+\frac14 \frac{t^2}{1+t^2} \right]_1^{\infty}\\
&= 
\frac{\pi}{2}.
\end{split}
\end{equation*}
This proves Lemma \ref{integral}.
\end{proof}

\section{Mean curvature computations} 
\label{meancurv}

In \cite[Equation 1.4]{escobar_92} we find the conformal
transformation formula for mean curvature. If $\tilde{g} = u^4 g$ then
the mean curvatures for $\tilde{g}$ and $g$ are related by
$$
\tilde{h} 
= 
\frac{2}{u^3} 
\left(\frac{\partial}{\partial \eta} + \frac{1}{2} h \right) u.
$$
Here $\frac{\partial}{\partial \eta}$ is the normal outward derivative
with respect to the metric $g$. If $r = d^g(p,\cdot)$ then 
$$
\frac{\partial}{\partial \eta} u
=
g(\nabla^g r, \nabla^g u)
$$
where $\nabla^g$ denotes the gradient and we get
$$
\tilde{h} 
= 
\frac{1}{u^3} 
\left(2 g(\nabla^g r, \nabla^g u) + hu \right) .
$$

If we apply this with our notation $g^k_{p,q} = G_k^4 g^k$ etc, then
we conclude
$$
H_k 
= 
\frac{1}{G_k^3} 
\left(2 g^k(\nabla^{g^k} r_k, \nabla^{g^k} G_k) + H^{g^k} G_k \right) .
$$
In our situation we have that the sphere $S_p^k(r)$ is close to a
round sphere of radius $r$ in flat $\mR^3$ for small $r$, therefore
$$
H^{g^k} = \frac{1}{r} + O(r) ,
$$
where the constant involved in the ordo term is independent of $k$
since $g^k$ tends to $g_{\infty}$.
We get 
$$
H_k 
= 
\frac{1}{G_k^3} 
\left(2 g^k(\nabla^{g^k} r_k, \nabla^{g^k} G_k) 
+ \left( \frac{1}{r} + O(r) \right) G_k \right) .
$$
We now insert the expansion of $G_k$ which we get from \eref{betak}, 
$$
G_k
= 
\frac{1}{r_k} + a_k + o(a_k) .
$$ 
Then
$$
\nabla^{g^k} G_k 
= 
 -\frac{1}{r_k^2} \nabla^{g^k} r_k +o(a_k)
$$
and
\begin{equation*}
\begin{split}
H_k 
&= 
\frac{1}{G_k^3} 
\left(2 g^k(\nabla^{g^k} r_k, \nabla^{g^k} G_k) 
+ \left( \frac{1}{r} + O(r) \right) G_k \right) \\
&=
\frac{1}{G_k^3} 
\left(2 g^k\left(\nabla^{g^k} r_k, 
\ -\frac{1}{r^2} \nabla^{g^k} r_k + o(a_k) \right) 
+ 
\left( \frac{1}{r} + O(r) \right) 
\left( \frac{1}{r} + a_k + o(a_k) \right) 
\right) \\
&=
\frac{1}{G_k^3} 
\left( - \frac{1}{r^2} + \frac{1}{r} (a_k +o(a_k)) + o(a_k) \right) \\ 
&=
\frac{1}{G_k^3 r^2} 
\left( - 1 + (a_k+o(a_k)) r + r^2  o(a_k)  \right),
\end{split}
\end{equation*}
since $g^k(\nabla^{g^k} r_k, \nabla^{g^k} r_k) = 1$.

\bibliographystyle{amsplain}
\bibliography{horizon}

\end{document}